\documentclass{svmult}

\usepackage{makeidx}         
\usepackage{graphicx}        
\usepackage{multicol}        

\usepackage{newtxtext}       %
\usepackage{newtxmath}       

\usepackage{mathtools}
\usepackage{graphicx}
\usepackage{epstopdf}
\usepackage{dsfont}
\usepackage[hidelinks]{hyperref}
\usepackage{enumerate}
\usepackage{algorithm}
\usepackage{bm}

\makeindex             

\usepackage{array}
\newcolumntype{R}{>{\raggedleft\let\newline\\\arraybackslash\hspace{0pt}}m{20pt}}

\begin{document}

\title*{An Adaptive Penalty Method for Inequality Constrained Minimization Problems}
\author{W. M. Boon and J. M. Nordbotten}
\institute{W. M. Boon \at 
	Department of Mathematics, KTH Royal Institute of Technology, Lindstedtsv\"agen 25, 11428, Stockholm, Sweden.
	\email{wietse@kth.se}
	\and
	J. M. Nordbotten \at Department of Mathematics, University of Bergen, Postboks 7803, 5020, Bergen, Norway \\
	Department of Civil and Environmental Engineering,
	Princeton University, E-208 E-Quad, Princeton, NJ, 08544, USA
	}

\maketitle

\abstract{
	The primal-dual active set method is observed to be the limit of a sequence of penalty formulations. Using this perspective, we propose a penalty method that adaptively becomes the active set method as the residual of the iterate decreases. The adaptive penalty method (APM) therewith combines the main advantages of both methods, namely the ease of implementation of penalty methods and the exact imposition of inequality constraints inherent to the active set method. The scheme can be considered a quasi-Newton method in which the Jacobian is approximated using a penalty parameter. This spatially varying parameter is chosen at each iteration by solving an auxiliary problem.
	}

\section{Introduction} \label{sec: Introduction}

	Inequality constrained minimization problems arise in a variety of applications, most prominently in contact problems in mechanics. To solve these problems, written as variational inequalities, a vast number of numerical methods exist and we refer the reader to \cite{kikuchi1988contact,suttmeier2009numerical,tremolieres2011numerical,wohlmuth2011variationally}, and references therein, for thorough expositions of such methods. This work concerns two seemingly unrelated families of numerical schemes, namely penalty methods (see e.g. \cite{carstensen1999adaptive,hansbo1999adaptive}) and the primal-dual active set method (see e.g. \cite{hintermuller2002primal,hueber2005primal}). 

	One of the main advantages of penalty methods is the ease of implementation. The penalty term can generally be incorporated as an addition to the original minimization problem in existing numerical software. Strictly speaking, however, the penalty term slightly alters the problem and the obtained solution may not satisfy the original constraints exactly. The active set method therefore forms an attractive alternative, as it does explicitly ensure that the solution complies to these constraints. Its disadvantage, however, is that the method typically requires an intrusive implementation in existing software and is prone to slow convergence.

	This work forms a link between these two families by proposing a penalty method that adaptively evolves to the primal-dual active set method. Depending on its interpretation, the scheme therefore belongs to both families. In particular, the scheme can be implemented as a penalty method and converges to the same solution as the active set method.

	Our starting point is the observation from \cite{hintermuller2002primal}, in which the primal-dual active set method is identified as a semi-smooth Newton method. We expand on this result by considering a regularization of the minimization problem to which the conventional Newton method can be applied. Instead of iterating until convergence, we introduce an adaptive removal of the regularization based on the residual in each iterative step. Thus, as the residual becomes smaller, the regularization decreases and the method is expected to convergence to the solution of the original problem. 

	The article proceeds as follows. Section~\ref{sec: Preliminary Concepts} introduces the family of constrained minimization problems of interest, the notational conventions, and a concise introduction to the primal-dual active set method and a specific class of penalty methods. 
	The main contribution of this work is presented in Section~\ref{sec: Adaptive Penalty Method}, namely an iterative scheme that adaptively combines the advantages of penalty and active set methods. Finally, Section~\ref{sec: Numerical Results} presents the numerical performance of the proposed scheme for a synthetic test case corresponding to a one-dimensional obstacle problem. 

\section{Problem Formulation and Solution Methods} \label{sec: Preliminary Concepts}
	On a given, open domain $\Omega \subset \mathbb{R}^n$, we consider the function space $V$. We assume $V$ is a reflexive Banach space with norm $\| \cdot \|$ and let $V^*$ denote its dual. Let $f \in V^*$ be a bounded linear functional and $A: V \to V^*$ a continuous, $V$-elliptic operator, i.e.
	\begin{align*}
		\langle f, v \rangle &\lesssim \| v \|, &
		\langle A u, v \rangle &\lesssim \| u \| \| v \|, & 
		\langle A v, v \rangle &\gtrsim  \| v \|^2, &
		\forall u, v &\in V.
	\end{align*}
	Here, $\langle \cdot, \cdot \rangle$ denotes the $V^* \times V$ duality pairing and the notation $a \lesssim b$ implies that a constant $C > 0$ exists such that $a \le C b$. For given $g \in V$, we consider the following constrained minimization problem:
	\begin{subequations} \label{eq: minimization problem}
		\begin{align}
			\min_{v \in V} J(v) &= \min_{v \in V} \frac{1}{2} \langle Av, v \rangle - \langle f, v \rangle \\
			\text{subject to } v &\le g
		\end{align}
	\end{subequations}

	Finding the minimizer $u \in V$ of problem \eqref{eq: minimization problem} is equivalent to solving either of the following two problems:
	\begin{multicols}{2}
	\noindent
	\emph{Primal formulation}: \\
	Find $u \in V$ such that
		\begin{subequations} \label{eq: primal}
			\begin{align}
					Au - f &\le 0, \\
					u - g &\le 0, \\
					\langle Au - f, u - g \rangle &= 0.
			\end{align}
		\end{subequations}
	\noindent
	\emph{Dual formulation}: \\
	Find $(u, \lambda) \in V \times V^*$ such that
		\begin{subequations} \label{eq: Dual}
			\begin{align} 
					Au - f + \lambda &= 0, \\
					\lambda &\ge 0, \\
					u - g &\le 0, \\
					\langle \lambda, u - g \rangle &= 0.
			\end{align}
		\end{subequations}
	\end{multicols}

	For both formulations, we can simplify the inequalities as well as the final equation into a single equation. For that purpose, we introduce the function $M: V^* \times V \to V^*$ given by
	\begin{align} \label{eq: def M}
		M(\phi, \varphi) := \phi - [\phi + c\varphi ]_+,
	\end{align}
	with $[\psi]_+ = \max\{0, \psi\}$ in the appropriate sense of elements of $V^*$. Moreover, $c: V \to V^*$ is the inverse Riesz map and we allow $c$ to include a scaling with a positive distribution. Clearly, we have
	\begin{align}
		M(\phi, \varphi) &= 0 & 
		&\Leftrightarrow & 
			\phi &\ge 0, \ 
			\varphi \le 0, \ 
			\langle \phi, \varphi \rangle = 0.
	\end{align}

	Thus, we can equivalently describe the primal formulation \eqref{eq: primal} by
	\begin{align} \label{eq: primal with M}
		M(f - Au, u - g) &= 0,
	\end{align}
	and the dual formulation \eqref{eq: Dual} by
	\begin{subequations} \label{eq: Dual with M}
		\begin{align}
			Au + \lambda &= f, \\
			M(\lambda, u - g) &= 0.
		\end{align}
	\end{subequations}

	To solve such problems numerically, we consider two families of iterative schemes, namely the active set method and penalty methods. We continue with a concise expsoition of these methods, presented in the following subsections, respectively.

\subsection{Primal-Dual Active Set Method}
	
	The primal-dual active set method uses the dual formulation \eqref{eq: Dual with M} and iteratively updates the set on which the constraint $u = g$ is imposed.	For the general problem \eqref{eq: minimization problem}, we define this \emph{active set} at iterate $k$ as
	\begin{subequations} \label{eq: def active region}
	\begin{align}
		\mathcal{A}^k := \{ x \in \Omega: \lambda^k(x) + c(u^k(x) - g(x)) > 0 \}.
	\end{align}
	
	In the case that $V$ is a piecewise linear finite element space defined by nodal evaluations at coordinates $x_i$, the active set $\mathcal{A}^k$ is defined by
	\begin{align}
		\mathcal{A}^k := \{ i: \lambda^k(x_i) + c(u^k(x_i) - g(x_i)) > 0 \}.
	\end{align}
	\end{subequations}
	Its complement on $\Omega$ is referred to as the \emph{inactive set}, denoted by $\mathcal{I}^k$. For brevity of notation, we introduce the indicator function $\mathds{1}_{\mathcal{A}}^k$ which is identity in $\mathcal{A}^k$ and zero otherwise. The indicator function $\mathds{1}_{\mathcal{I}}^k$ is defined analogously.
	For a given active set $\mathcal{A}^k$, the primal-dual active set method then solves the following system of equations 
			\begin{align*}
				\begin{bmatrix}
					A & I \\
					-\mathds{1}_{\mathcal{A}}^k c & \mathds{1}_{\mathcal{I}}^k
				\end{bmatrix}
				\begin{bmatrix}
					u^{k + 1} \\
					\lambda^{k + 1}
				\end{bmatrix} 
				&= 
				\begin{bmatrix}
					f \\
					-\mathds{1}_{\mathcal{A}}^k c g
				\end{bmatrix}
			\end{align*}	
	We simplify this system by substituting $\lambda^{k + 1} = f - Au^{k + 1}$ from the first row into the second, giving us Algorithm~\ref{alg: ASM}.
	\begin{algorithm}[ht]
		\caption{Active Set Method}
		\label{alg: ASM}
		\begin{enumerate}[(i)]
			\item Set $k = 0$ and initialize $u^0$.
			\item Compute $\mathcal{A}^k$ using \eqref{eq: def active region}.
			\item Solve for $u^{k+1}$:
			\begin{align} \label{eq: Active Set Method}
				(\mathds{1}_{\mathcal{I}}^k A + \mathds{1}_{\mathcal{A}}^k c) u^{k + 1}
				&= \mathds{1}_{\mathcal{I}}^k f + \mathds{1}_{\mathcal{A}}^k c g.
			\end{align}
			\item Stop if converged, else increment $k$ and return to (ii).
		\end{enumerate}
	\end{algorithm}


\subsection{Penalty Method}

	The defining attribute of penalty methods is the modification of the formulation by introducing a term which penalizes the solution $u$ if it is outside the admissible set \cite{grossmann2007numerical}. To be precise, we introduce a penalty parameter $\rho \ge 0$ and an operator $\Pi_{\rho}: V \to V^*$ to modify the primal formulation \eqref{eq: primal} to: \\
	Find $u \in V$ such that
		\begin{align} \label{eq: penalty method}
			A u - f + \Pi_{\rho} u &= 0.
		\end{align}

	We use the convention that a smaller value of $\rho$ corresponds to a stricter penalization. 

	It is advantageous to choose the penalty operator $\Pi_{\rho}$ sufficiently smooth in order to apply the Newton method.
	We consider a particular choice of $\Pi_{\rho}$ obtained from a regularization of the problem \eqref{eq: primal with M}. For that purpose we use \cite{chen1995smoothing} and let $[\cdot]_\rho$ be the smooth approximation of $[\cdot]_+$ given by
		\begin{align*}
			[\phi]_\rho &:= \phi + \rho \log (1 + \exp( - \phi / \rho)), &
			[\phi]_\rho' &= (1 + \exp(- \phi / \rho))^{-1}.
		\end{align*}
	It is important to note that this function and its derivative have the following properties for all $\phi \in V^*$:
	\begin{align} \label{eq: smoothed ramp}
		\lim_{\rho \downarrow 0} [\phi]_\rho &= [\phi]_+, &
		\lim_{\rho \downarrow 0} [\phi]_\rho'
		= \lim_{\rho \downarrow 0} \frac{d}{d \phi} [\phi]_\rho 
		&= \mathds{1}_{\phi > 0}.
	\end{align}
	
	Using this operator, we define the regularization of $M$ as
	\begin{align*}
		M_{\rho}(\phi, \varphi) &:= \phi - [\phi + c\varphi]_\rho.
	\end{align*}
	In turn, a regularization of the primal formulation \eqref{eq: primal with M} arises:
		\begin{align} \label{eq: primal regularized}
			-M_\rho(f - Au, u - g) 
			= Au - f + [f - Au + c(u - g) ]_\rho 
			= 0
		\end{align}

	Note that this corresponds to setting
		$\Pi_\rho u := [f - Au + c(u - g) ]_\rho$
	in equation \eqref{eq: penalty method} and we conclude that the regularized formulation \eqref{eq: primal regularized} has the structure of a penalty method.
	
	Applying this regularization to the dual formulation \eqref{eq: Dual with M}, we similarly obtain
	\begin{subequations} \label{eq: dual regularized}
		\begin{align}
			Au + \lambda &= f, \\
			M_\rho(\lambda, u - g) &= 0.
		\end{align}
	\end{subequations}

	Due to the smoothness of $M_\rho$, the Newton method becomes an attractive solution strategy and we therefore apply this method to the regularized primal problem \eqref{eq: primal regularized}. This leads us to the penalty method presented as Algorithm~\ref{alg: Newton} below. We remark that $\alpha_\rho^k$ is interpreted as a diagonal operator here.

\begin{algorithm}[ht]
	\caption{Penalty Method}
	\label{alg: Newton}
	\begin{enumerate}[(i)]
		\item Set $k = 0$ and initialize $u^0$.
		\item Compute $\alpha_\rho^k = [f - Au^k + c(u^k - g)]_\rho'$.
		\item Solve for $\delta u$:
		\begin{align} \label{eq: primal Newton}
			( (I - \alpha_\rho^k)A + \alpha_\rho^k c) \delta u &= M_{\rho}(f - Au^k, u^k - g).
		\end{align}
		and set $u^{k + 1} = u^{k} + \delta u$.
		\item Stop if converged, else increment $k$ and return to (ii).
	\end{enumerate}
\end{algorithm}

	We make two observations concerning Algorithm~\ref{alg: Newton}, presented as two lemmas. First, we show an equivalent derivation using the dual formulation \eqref{eq: dual regularized} and secondly, we note the behavior of the scheme as the penalty parameter tends to zero.
	\begin{lemma}
		Applying the Newton method to the regularized dual formulation \eqref{eq: dual regularized} equivalently leads to Algorithm~\ref{alg: Newton}.
	\end{lemma}
	\begin{proof}

		Let us linearize the dual formulation \eqref{eq: dual regularized} around the previous iterate $(u^k, \lambda^k)$. Applying the Newton method leads to
		\begin{align*}
			\begin{bmatrix}
				A & I \\
				\frac{\partial}{\partial u} M_{\rho}(\lambda^k, u^k - g) &
				\frac{\partial}{\partial \lambda} M_{\rho}(\lambda^k, u^k - g)
			\end{bmatrix}
			\begin{bmatrix}
				\delta u \\
				\delta \lambda
			\end{bmatrix}
			= - \begin{bmatrix}
			Au^k + \lambda^k - f \\
			M_\rho(\lambda^k, u^k - g)
			\end{bmatrix}
		\end{align*}
		
		By introducing $\alpha_\rho^k = [\lambda^k + c(u^k - g)]_\rho'$, we specify the derivatives and rewrite:
		\begin{align}
			\begin{bmatrix}
				A & I \\
				-\alpha_\rho^kc &
				I - \alpha_\rho^k
			\end{bmatrix}
			\begin{bmatrix}
				\delta u \\
				\delta \lambda
			\end{bmatrix}
			&= 
			- \begin{bmatrix}
			Au^k + \lambda^k - f\\
			M_\rho(\lambda^k, u^k - g) 
			\end{bmatrix}. \label{eq: dual Newton}
		\end{align}
		
		Next, we note that $\lambda^k = f - A u^k$ for $k>0$, giving us $\delta \lambda = -A \delta u$ from the first row. Substituting this into the second row gives us
		\begin{align*}
			(-\alpha_\rho^kc - (I - \alpha_\rho^k)A) \delta u
			= -M_\rho(f - A u^k, u^k - g).
		\end{align*}
		Negation of this equation gives us \eqref{eq: primal Newton}, thereby concluding the proof. 
		\end{proof}


	\begin{lemma}
		Algorithm~\ref{alg: Newton} is equivalent to Algorithm~\ref{alg: ASM} in the limit $\rho \downarrow 0$.
	\end{lemma}
\begin{proof}

	By \eqref{eq: smoothed ramp}, the limit $\rho \downarrow 0$ gives us $\alpha_\rho^k \to \mathds{1}_\mathcal{A}^k$, i.e. the indicator function of $\mathcal{A}^k$. Moreover, the operator $M_{\rho}(\cdot, \cdot)$ on the right-hand side becomes $M(\cdot, \cdot)$. Equation \eqref{eq: primal Newton} then becomes
		\begin{align} 
			( \mathds{1}_{\mathcal{I}}^k A + \mathds{1}_{\mathcal{A}}^k c) \delta u &= M(f - Au^k, u^k - g) 
			= \mathds{1}_{\mathcal{I}}^k (f - Au^k) - \mathds{1}_{\mathcal{A}}^k c(u^k - g)
		\end{align}
	Addition of $( \mathds{1}_{\mathcal{I}}^k A + \mathds{1}_{\mathcal{A}}^k c) u^k$ to both sides of the equation gives us \eqref{eq: Active Set Method}.
\end{proof}

\section{The Adaptive Penalty Method} \label{sec: Adaptive Penalty Method}

In the previous section, we have made two observations. First, introducing a penalty parameter $\rho$ leads to a regularized problem on which the Newton method can be applied. This method is known to be converge (locally) to the solution of the regularized problem. Secondly, as $\rho$ tends to zero, the penalty method becomes equivalent to the active set method, which respects the inequality constraint of \eqref{eq: minimization problem} exactly. The next step is to combine these two advantages into a single iterative method.

With this goal in mind, we modify the penalty method by letting $\rho$ be a spatially varying function on $\Omega$. This allows us to adaptively remove the penalization in regions where the solution is sufficiently accurate. We achieve this by constructing the penalty function $\rho$ as a regularization of the residual. Let us therefore introduce the following differential equation for $\rho$:
\begin{subequations} \label{eq: rho}
	\begin{align}
		\rho - \epsilon \Delta \rho &= \gamma | M(f - Au, u - g) | &\text{ in } & \Omega, \\
		\bm{n} \cdot \nabla \rho &= 0 & \text{ on } & \partial \Omega.
	\end{align}
\end{subequations}
Here, $|\cdot|$ denotes the absolute value, $\bm{n}$ is the outward unit normal vector on $\partial \Omega$ and $\epsilon, \gamma$ are chosen, nonnegative constant parameters. For simplicity, we limit our exposition to these two tuning parameters.

By elliptic regularity of \eqref{eq: rho}, the penalization $\rho$ will tend to zero as the residual becomes smaller. We exploit this property and propose Algorithm~\ref{alg: Adaptive Penalty Method}, which we refer to as the Adaptive Penalty Method (APM).
\begin{algorithm}[ht]
\caption{Adaptive Penalty Method}
\label{alg: Adaptive Penalty Method}
\begin{enumerate}[(i)]
	\item Set $k = 0$ and initialize $u^0$.
	\item Solve \eqref{eq: rho} for the regularization parameter $\rho$ with data $u = u^k$.
	\item Compute $\alpha_\rho^k = [f - Au^k + c(u^k - g)]_\rho'$.
	\item Solve for $\delta u$:
	\begin{align}
		( (I - \alpha_\rho^k)A + \alpha_\rho^k c) \delta u &= M(f - Au^k, u^k - g) .
	\end{align}
	and set $u^{k + 1} = u^k + \delta u$.
	\item Stop if converged, else increment $k$ and return to (ii).
\end{enumerate}
\end{algorithm}

It is important to note that the exact solution to the auxiliary problem \eqref{eq: rho} is not our main priority. Thus, in order to reduce computational cost, it will suffice to use an approximate solution in step (ii) with the use of a coarse solve or multi-grid cycle.

Algorithm~\ref{alg: Adaptive Penalty Method} can be interpreted in a variety of ways. First, the scheme is a quasi-Newton method on \eqref{eq: primal with M} in which the Jacobian gets approximated more accurately as the solution converges. The accuracy of the Jacobian adaptively depends on the residual, hence the chosen name.

Alternatively, the algorithm can be considered a warm-start that gradually behaves like the active set method in convergence. The advantage in this context is that no invasive implementations are necessary to switch from the warm-start to the active set method.

Thirdly, the choice of $\gamma = 0$ results in $\rho = 0$ in all iterations and the scheme is effectively reduced to Algorithm~\ref{alg: ASM}. In that sense, this construction serves as a generalization of primal-dual active set method. This is an advantage in case optimal parameter values are difficult to find, since the scheme can easily be reduced to the active set method without requiring additional, numerical implementation.

Other extreme choices of the parameters lead to different behaviors of the proposed scheme. A large value of $\gamma$, for example, results in a slower decrease of the regularization parameter and therewith, a slower convergence to the solution. On the other hand, setting $\epsilon = 0$ removes the diffusion in \eqref{eq: rho} which typically results in sporadic behavior of the scheme and possibly, loss of convergence. However, choosing a too large value for $\epsilon$ makes the diffusion term dominate which results in a spatially uniform penalty parameter. This is disadvantageous since it leads to unnecessarily poor approximations of the Jacobian in regions where the solution is close to exact. 

\section{Numerical Results} \label{sec: Numerical Results}

In this section, we test the numerical performance of the adaptive penalty method using a synthetic test case. Let us consider an obstacle problem on $\Omega = (0, 1)$. We aim to find $u \in H_0^1(\Omega)$ that weakly satisfies
\begin{subequations}
	\begin{align}
		- \Delta u &\le f, & f(x) &:= 10, \\
		u &\le g, & g(x) &:= 0.2(1 + \mathds{1}_{x > 0.25} + \mathds{1}_{x > 0.5} + \mathds{1}_{x > 0.75}), \\
		\langle \Delta u + f, u - g \rangle &= 0, & \text{in }& \Omega, \\
		u &= 0, & \text{on }& \partial \Omega.
	\end{align}
\end{subequations}

We set the scaling in the Riesz operator $c$ to unity and iterate until the Euclidean norm of the residual is below a tolerance level of 1e-10. In the numerical experiments, we have not observed significant sensitivities of the scheme with respect to $\epsilon$ and therefore limit this exposition to $\epsilon = 1$.

As remarked in the previous section, an interesting variant of the method arises if the penalty parameter is approximated, instead of solving \eqref{eq: rho} exactly. To explore this variant, we perform a solve on a coarse mesh of 16 elements and interpolate back to the original mesh. We compare three methods, namely the primal-dual active set method (Algorithm~\ref{alg: ASM}), the adaptive penalty method (Algorithm~\ref{alg: Adaptive Penalty Method}) introduced in Section~\ref{sec: Adaptive Penalty Method}, and its variant with a coarse solve. The results are shown in Table~\ref{tab: Iterations}.

\begin{table}[ht]
	\caption{Number of iterations necessary to obtain the desired accuracy for the active set method (ASM), the adaptive penalty method with the penalization $\rho$ solved exactly (APM), and on a coarse mesh (C-APM). The proposed schemes obtain the same solution as the active set method in fewer iterations.}
	\label{tab: Iterations}
	\centering
	\footnotesize
	\begin{tabular}{r|lR|RRRR|RRRR}
	\hline

	\hline
	 & \multicolumn{2}{c|}{\textbf{ASM}} & \multicolumn{4}{c|}{\textbf{APM}} & \multicolumn{4}{c}{\textbf{C-APM}} \\ 
	\hline
	$1/h$  & $\gamma$ & 0  & 0.1 & 1 & 10 & 100  & 0.1 & 1 & 10 & 100 \\
	\hline
 		256  & &   52 &  18 &   9 &  13 &  29 &  18 &   9 &  13 &  31 \\
 		512  & &  103 &  33 &  10 &  12 &  33 &  32 &  10 &  13 &  32 \\
 		1024 & &  206 &  69 &  33 &  16 &  45 &  68 &  32 &  14 &  47 \\
 		2048 & &  411 & 162 &  52 &  18 &  50 & 160 &  50 &  18 &  50 \\
 		4096 & &  820 & 382 &  73 &  22 &  53 & 378 &  69 &  26 &  56 \\
 		8192 & & 1639 & 876 & 107 &  31 &  48 & 864 &  99 &  36 &  54 \\
	\hline
	\end{tabular}
\end{table}

From the numerical experiment, we observe that the Adaptive Penalty Method requires significantly fewer iterations than the primal-dual active set method for this problem. As discussed, small values of $\gamma$ cause the scheme to behave like the active set method and this can be observed in the iteration numbers. Moreover, the number of iterations appear robust with respect to the grid size for the largest choices of $\gamma$.

The results from C-APM indicate that the exact evaluation of $\rho$ can be avoided, in practice. This makes the scheme attractive for larger linear systems in terms of computational cost, since there is no need to solve an additional linear system during each iteration. 

To conclude, the proposed Adaptive Penalty Method rapidly converges to the same solution as the primal-dual active set method, which satisfies the constraints of the original problem exactly. The scheme is easily implementable as a penalty method or as a quasi-Newton scheme in existing software. To reduce computational cost, the penalty parameter can be approximated using a coarse solve, without significantly affecting the convergence of the method. 

\begin{acknowledgement}
	This work was partially supported by Norwegian Research Council grant 233736.
\end{acknowledgement}
\bibliographystyle{spmpsci}
\bibliography{biblio}

\end{document}